\documentclass[11pt]{amsart}
\usepackage{latexsym,amsmath,amssymb,amsthm,mathrsfs,amscd,enumerate,latexsym,array}
\usepackage[all]{xy}

\theoremstyle{definition}
\newtheorem{Unity}{Unity}[section]
\newtheorem{dfn}[Unity]{Definition}
\newtheorem{rmk}[Unity]{Remark}
\newtheorem{ntt}[Unity]{Notation}
\newtheorem{exam}[Unity]{Example}

\theoremstyle{plain}
\newtheorem{thm}[Unity]{Theorem}
\newtheorem{prop}[Unity]{Proposition}
\newtheorem{conj}[Unity]{Conjecture}
\newtheorem{lem}[Unity]{Lemma}

\newtheorem{claim}[Unity]{Claim}

\begin{document}

\title[Fano manifolds whose Chern characters satisfy some positivity conditions]{Fano manifolds whose Chern characters satisfy some positivity conditions}
\author{Taku Suzuki}
\address{Cooperative Faculty of Education, Utsunomiya University, 350 Mine-machi, Utsunomiya, Tochigi, 321-8505, Japan}
\email{taku.suzuki@cc.utsunomiya-u.ac.jp}
\thanks{The author is supported by JSPS KAKWNHI Grant Number 21K13764.}

\maketitle

\begin{abstract}
In this paper, we investigate Fano manifolds whose Chern characters satisfy some positivity conditions.
We prove that such manifolds admit long chains of higher order minimal families of rational curves and are covered by higher rational varieties.
\end{abstract}

\section{Introduction}

Fano manifolds, namely, smooth complex projective varieties having (positive dimension and) positive first Chern class, are fundamental subjects in algebraic geometry.
In \cite{Mo}, Mori showed that any Fano manifold is covered by rational curves.
Since then, rational curves have played an important role in the study of Fano manifolds.

In this paper, we discuss Fano manifolds whose Chern characters satisfy some positivity conditions.
Such manifolds have been investigated by many authors and are expected to have stronger versions of several properties of Fano manifolds.

In \cite{JS}, de Jong and Starr investigated Fano manifolds with nef second Chern character.
They showed that such manifolds are covered by rational surfaces under the assumption that the pseudo-index is at least $3$, where the pseudo-index is defined as the minimal anticanonical degree of rational curves.

In \cite{AC1}, Araujo and Castravet  investigated the Chen characters of minimal families of rational curves.
For a Fano manifold $X$, a proper irreducible component $H$ of the scheme parametrizing rational curves on $X$ through a fixed general point is called a \textit{minimal family of rational curves}. 
They gave a formula for the Chern characters of $H$ in terms of the Chern characters of $X$.
As a consequence, they proved that Fano manifolds having positive second Chern character and nef third Chern character are covered by rational varieties of dimension $3$ under some extra assumptions.

In \cite{AC2}, Araujo and Castravet classified Fano manifolds of large index having positive second Chern character, where the index is defined as the largest integer dividing the anticanonical divisor.
Furthermore, in \cite{ABCJMMTV}, Araujo, Beheshti, Castravet, Jabbusch, Makarova, Mazzon, Taylor, and Viswanathan classified rational homogeneous spaces of Picard number $1$ having positive second Chern character and Fano manifolds of large index having positive second and third Chern characters.
Their results show that positivity conditions of Chern characters are restrictive, and they proposed the following conjecture:

\begin{conj}[{\cite[Question 1.2]{AC1} and \cite[Conjecture 1.7]{ABCJMMTV}}]\label{Conj1}
Let $X$ be a Fano manifold of dimension $n$.
If ${\rm ch}_k(X)$ is positive for every $k \le n$, then $X$ is isomorphic to $\mathbb{P}^n$.
More strongly, if ${\rm ch}_k(X)$ is positive for every $k \le \lceil {\rm log}_2 (n+2) \rceil$, then $X$ is isomorphic to $\mathbb{P}^n$.
\end{conj}

In \cite{Su2}, the author introduced \textit{higher order minimal families of rational curves}.
For a Fano manifold $X$, if a minimal family $H_1$ of rational curves on $X$ is also a Fano manifold, then we can take a minimal family $H_2$ of rational curves on $H_1$, which we call a \textit{second order minimal family}.
Similarly we define an \textit{$i$-th order minimal family} $H_i$.
In addition, we define $\underline{N}_X$ (resp.\ $\overline{N}_X$) as the smallest (resp.\ largest) number $i$ such that $X$ admits an $i$-th order minimal family $H_i$ which is not a Fano manifold.
We note that $1 \le \underline{N}_X \le \overline{N}_X \le {\rm dim}X$ for any Fano manifold $X$.
The author and Nagaoka computed the Chen characters of $H_i$ and obtained the following theorem:

\begin{thm}[{\cite[Theorem 5.1]{Su2} and \cite[Theorem 5.1]{Na}}]\label{Thm1}
Let $X$ be a Fano manifold and $m$ a positive integer.
Assume that ${\rm ch}_k(X)$ is nef for every $2 \le k \le m$.
Assume also that $X$ admits a minimal family $H_1$ of rational curves having dimension at least $m^2-m-1$.
Then the following statements hold:
\begin{enumerate}
\setlength{\itemsep}{3pt}
\item $\overline{N}_X \ge m$.
\item $X$ is covered by rational varieties of dimension $m$, except possibly if either $H_1$ is isomorphic to a projective space, a quadric hypersurface, or $X$ admits an $i$-th order minimal family $H_i$ which is isomorphic to a projective space for some $2 \le i \le m-1$.
\end{enumerate}
\end{thm}

\begin{rmk}
In Theorem \ref{Thm1}(1), $\underline{N}_X \ge m$ also holds under the stronger assumption that every minimal family has dimension at least $m^2-m-1$.
In Theorem \ref{Thm1}(2), moreover if $H_1$ parametrizes rational curves of degree $1$ with respect to some ample line bundle on $X$, then $X$ is covered by projective spaces of dimension $m$.
\end{rmk}

The aim of this paper is to remove the extra assumption and the exception from Theorem \ref{Thm1} by replacing the assumption for the Chern characters of $X$.
The main theorems of this paper are as follows (see Definition \ref{Def1} for positivity condions):

\begin{thm}[{Theorem \ref{Thm4}}]\label{Thm2} 
Let $X$ be a Fano manifold and $m$ a positive integer.
Assume that
$${\rm ch}_{k}(X) \ge \frac{m+1}{k!}$$
for every $1 \le k \le m$.
Then the following statements hold:
\begin{enumerate}
\setlength{\itemsep}{3pt}
\item $\underline{N}_X \ge m$.
\item $X$ is covered by rational varieties of dimension $m$.
Moreover, if $X$ is covered by rational curves of degree $1$ with respect to some ample line bundle on $X$, then $X$ is covered by projective spaces of dimension $m$.
\end{enumerate}
\end{thm}

\begin{thm}[{Theorem \ref{Thm5}}]\label{Thm3}
Let $X$ be a Fano manifold and $m$ a positive integer.
Assume that $X$ is covered by rational curves of degree $1$ with respect to some ample line bundle $L$ on $X$ and
$${\rm ch}_{k}(X) \ge \frac{2m+1-2^k}{k!} c_1(L)^k$$
for every $1 \le k \le m$.
Then the following statements hold:
\begin{enumerate}
\setlength{\itemsep}{3pt}
\item $\overline{N}_X \ge m$.
Moreover, if every minimal family parametrizes rational curves of degree $1$ with respect to $L$, then $\underline{N}_X \ge m$.
\item $X$ is covered by rational varieties of dimension $m$ and by projective spaces of dimension $m-1$.
Moreover, under the stronger assumption that
$${\rm ch}_{k}(X) \ge \frac{2m+2-2^k}{k!} c_1(L)^k$$
for every $1 \le k \le m$, $X$ is covered by projective spaces of dimension $m$.
\end{enumerate}
\end{thm}

\begin{rmk}
The projective space $\mathbb{P}^n$ satisfies the assumption of Theorem \ref{Thm2} for $m=n$  because
$${\rm ch}_k (\mathbb{P}^n) = \frac{n+1}{k!} c_1\bigl(\mathscr{O}(1)\bigl)^k$$
for every $1 \le k \le n$.
In fact, $\underline{N}_{\mathbb{P}^n} = \overline{N}_{\mathbb{P}^n} = n$ (see Example \ref{Exam1}(1)).
On the other hand, the quadric hypersurface $Q^n$ satisfies the assumption of Theorem \ref{Thm3} for $m=\lceil \frac{n}{2} \rceil$ because
$${\rm ch}_k (Q^n) = \frac{n+2-2^k}{k!} c_1\bigl(\mathscr{O}(1)\bigl)^k$$
for every $1 \le k \le n$.
In fact, $\underline{N}_{Q^n} = \overline{N}_{Q^n} = \lceil \frac{n}{2} \rceil$ (see Example \ref{Exam1}(2)) and it is covered by projective spaces of dimension $\lfloor \frac{n}{2} \rfloor$.
We also remark that, under the assumption of Theorem \ref{Thm3}, ${\rm ch}_k(X)$ is positive for $k <{\rm log}_2 (2m+1)$ but may possibly not be positive for $k \ge {\rm log}_2 (2m+1)$.
\end{rmk}

This paper is organized as follows. 
In Section 2, we review positivity conditions for cycles, minimal families of rational curves, and higher order minimal families of rational curves.  
In Section 3, we discuss computations of the Chern characters of higher order minimal families. 
In Section 4, we prove Theorems \ref{Thm2} and \ref{Thm3}.

\section{Preliminaries}

\subsection{Positivity conditions for cycles}

\begin{dfn}\label{Def1}
Let $X$ be a manifold and $k$ a non-negative integer.
We denote by $N^k(X)$ the group of cycles of codimension $k$ on $X$ modulo numerical equivalence, and set $N^k(X)_{\mathbb{R}}:=N^k(X) \otimes \mathbb{R}$.
For $\alpha \in N^k(X)_{\mathbb{R}}$ and $r \in \mathbb{R}_{\ge 0}$, we write $\alpha \ge r$ (resp.\ $\alpha >r$) when $\alpha \cdot \beta \ge r$ (resp.\ $\alpha \cdot \beta > r$) holds for every non-zero effective integral cycle $\beta$ of dimension $k$ on $X$.
In particular, when $\alpha \ge 0$ (resp.\ $\alpha >0$), we say that $\alpha$ is \textit{nef} (resp.\ \textit{positive}).
For $\alpha_1, \alpha_2 \in N^k(X)_{\mathbb{R}}$, we write $\alpha_1 \ge \alpha_2$ (resp.\ $\alpha_1 > \alpha_2$) when $\alpha_1 - \alpha_2 \ge 0$ (resp.\ $\alpha_1 - \alpha_2 > 0$).
\end{dfn}

\subsection{Minimal families of rational curves}

We refer to \cite{AC1} and \cite[I and II]{Ko} for basic theory of families of rational curves.

\begin{dfn}\label{Def2}
Let $X$ be a Fano manifold and $x$ a general point of $X$.
We denote by $\text{RatCurves}^n(X,x)$ the scheme parametrizing rational curves on $X$ through $x$ (see \cite[II.2]{Ko}).
A proper irreducible component $H$ of $\text{RatCurves}^n(X,x)$ is called a \textit{minimal family} of rational curves on $X$ through $x$.
There always exists such a family, for instance, an irreducible component of $\text{RatCurves}^n(X,x)$ parametrizing rational curves having minimal degree with respect to some fixed ample line bundle on $X$.
It is known that any minimal family $H$ is a smooth projective variety and admits a finite morphism $\tau : H \rightarrow \mathbb{P}(T_x X^{\vee})$ which is called the \textit{tangent map} (see \cite[Theorems 3.3 and 3.4]{Ke}). 
We note that $\tau$ is birational onto its image (see \cite{HM}) and sends a curve which is smooth at $x$ to its tangent direction.
Set $L:=\tau ^* \mathscr{O}(1)$. 
Then we call the pair $(H,L)$ a \textit{polarized minimal family} of rational curves on $X$ through $x$. 
Let $U$ be the universal family of $H$ and let $\pi : U \rightarrow H$, $e: U \rightarrow X$ be the associated morphisms.
We define a linear map
$$T: N^k(X)_{\mathbb{R}} \rightarrow N^{k-1}(H)_{\mathbb{R}},\ \ \alpha \mapsto \pi_* e^* \alpha.$$
\end{dfn}

\begin{lem}\label{Lem1}
Let $X$, $H$, $L$, and $T$ be as in Definitions \ref{Def2}.
For $D \in N^1(X)_{\mathbb{R}}$, let $a$ be the intersection number of $D$ and a curve parametrized by $H$.
Let $k, m$ be positive integers, $\alpha \in N^k(X)_{\mathbb{R}}$, and $r \in \mathbb{R}_{\ge 0}$.
Then the following statements hold:
\begin{enumerate}
\setlength{\itemsep}{3pt}
\item $T(D^m)=a^m c_1(L)^{m-1}$.
\item $T(\alpha \cdot D^m)=a^m T(\alpha) \cdot c_1(L)^m$.
\item ${\rm dim}H = T\bigl(c_1(X)\bigr)-2$. 
\item If $\alpha \ge r$ (resp.\ $\alpha >r$), then $T(\alpha) \ge r$ (resp.\ $T(\alpha) >r$). 
\end{enumerate}
\end{lem}

\begin{proof}
(1), (2), and (3) have been proved in \cite[Lemma 2.7]{AC1} and \cite[Lemma 2.10]{Su2}.
If $\beta$ is an effective integral cycle on $H$ of dimension $k-1$, then $T(\alpha) \cdot \beta = \alpha \cdot (e_{*} {\pi}^{*} \beta)$ by the projection formula and $e_{*} {\pi}^{*} \beta$ is a non-zero effective cycle on $X$ (see \cite[Lemma 2.7]{AC1}).
This implies (4).
\end{proof}

\begin{lem}\label{Lem2}
Let $X$, $x$, $H$, and $L$ be as in Definitions \ref{Def2} and we denote by $H^{\circ}$ the subvariety of $H$ parametrizing curves which are smooth at $x$.
\begin{enumerate}
\setlength{\itemsep}{3pt}
\item Assume that there is a proper rational variety $Z \subset H^{\circ}$ of dimension $k$.
Then there is a proper rational variety $x \in Y \subset X$ of dimension $k+1$.
\item Assume that there is a proper variety $Z \subset H^{\circ}$ such that $(Z,L|_Z)\cong (\mathbb{P}^k,\mathscr{O}(1))$. 
Then there is a generically injective morphism $f: (\mathbb{P}^{k+1},p) \rightarrow (X,x)$ which maps lines through $p$ birationally to curves parametrized by $H$. 
\item Assume that, for a general point $h \in H$, there is a proper variety $h \in Z_h \subset H$ such that $(Z_h,L|_{Z_h})\cong (\mathbb{P}^k,\mathscr{O}(1))$.
Then there is a generically injective morphism $f: (\mathbb{P}^{k+1},p) \rightarrow (X,x)$ which maps lines through $p$ birationally to curves parametrized by $H$. 
\end{enumerate}
\end{lem}

\begin{proof}
(1) follows from \cite[II Corollary 2.12 and V Proposition 3.7.5]{Ko}.
(2) is \cite[Lemma 2.3]{AC1}.
To prove (3), we show $Z_h \subset H^{\circ}$ for some point $h \in H$.
We assume by contradiction that $Z_h \not\subset H^{\circ}$ for a general point $h \in H$.
Then, since $H \backslash H^{\circ}$ is finite (see \cite[Theorem 3.3]{Ke}), there is a point $h_0 \in H \backslash H^{\circ}$ such that $h_0 \in Z_h$ for a general point $h \in H$. 
This means that $H$ is covered by rational curves of $L$-degree $1$ through the fixed point $h_0$.
Thus $(H,L)$ must be isomorphic to $(\mathbb{P}^m, \mathscr{O}(1))$.
However, this implies $H^{\circ}=H$ by \cite[Corollary 2.8]{Ar}.
\end{proof}

\subsection{Higher order minimal families of rational curves}

\begin{dfn}\label{Def3}
Let $X$ be a Fano manifold and take a minimal family $H_1$ of rational curves on $X$ through a fixed general point.
If $H_1$ is also a Fano manifold, we can take a minimal family $H_2$ of rational curves on $H_1$ through a fixed general point, which we call a \textit{second order minimal family}.
Moreover, if $H_2$ is also a Fano manifold, we can take a minimal family $H_3$ of rational curves on $H_2$ through a fixed general point, which we call a \textit{third order minimal family}.
Similarly we define an \textit{$i$-th order minimal family} $H_i$.
We denote the chain of these higher order minimal families by 
$$X \vdash H_1 \vdash H_2 \vdash H_3 \vdash \cdots \vdash H_i.$$
We define $\underline{N}_X$ (resp.\ $\overline{N}_X$) as the smallest (resp.\ largest) number $i$ such that $X$ admits an $i$-th order minimal family $H_i$ which is not a Fano manifold.
\end{dfn}

\begin{exam}\label{Exam1}
We give some examples of computations of $\underline{N}_X$ and $\overline{N}_X$.
See \cite[1.4]{Hw} for examples of minimal families of rational curves.
\begin{enumerate}
\setlength{\itemsep}{5pt}
\item The projective space $\mathbb{P}^n$ has the following chain:
$$\mathbb{P}^n \vdash \mathbb{P}^{n-1} \vdash \mathbb{P}^{n-2} \vdash \cdots \vdash \mathbb{P}^2 \vdash \mathbb{P}^1 \vdash {\rm pt},$$
so we have $$\underline{N}_{\mathbb{P}^n} = \overline{N}_{\mathbb{P}^n} =n.$$
\item The quadric hypersurface $Q^n$  has the following chain:
\begin{eqnarray*}
&Q^{2m} \vdash Q^{2m-2} \vdash Q^{2m-4} \vdash \cdots \vdash Q^4 \vdash Q^2 \vdash {\rm pt},\\
&Q^{2m+1} \vdash Q^{2m-1} \vdash Q^{2m-3} \vdash \cdots \vdash Q^3 \vdash Q^1 \vdash {\rm pt},
\end{eqnarray*}
so we have $$\underline{N}_{Q^n} = \overline{N}_{Q^n} = \Bigl\lceil \frac{n}{2} \Bigr\rceil.$$
\item The Grassmannian $G(k,m)$ has the following two chains:
$$G(k,m) \vdash \mathbb{P}^{k-1} \times \mathbb{P}^{m-k-1} \vdash \begin{cases} \mathbb{P}^{k-2} \vdash \mathbb{P}^{k-3} \vdash \cdots \vdash \mathbb{P}^1 \vdash {\rm pt} \\ \mathbb{P}^{m-k-2} \vdash \mathbb{P}^{m-k-3} \vdash \cdots \vdash \mathbb{P}^1 \vdash {\rm pt} \end{cases},$$
so we have $$\underline{N}_{G(k,m)} =  {\rm min} \{ k, m-k\},\ \ \overline{N}_{G(k,m)} = {\rm max} \{ k, m-k\}.$$
\end{enumerate}
\end{exam}

\begin{rmk}
For any Fano manifold $X$ of dimension $n$, clearly $1 \le \underline{N}_X \le \overline{N}_X \le n$ holds.
In addition, if $\overline{N}_X = n$, then $X$ is isomorphic to $\mathbb{P}^n$ by Cho, Miyaoka, and Shepherd-Barron's characterization of projective spaces (see \cite{CMSB}).
Furthermore, the author believe that Conjecture \ref{Conj2} below holds.
A special case of this conjecture has been proved in \cite{Su1}.
\end{rmk}

\begin{conj}\label{Conj2}
Let $X$ be a Fano manifold of dimension $n$ and Picard number $1$.
\begin{enumerate}
\item If $\underline{N}_X \ge \lceil \frac{n}{2} \rceil$, then $X$ is isomorphic to either $\mathbb{P}^n$ or $Q^n$.
\item If $\overline{N}_X \ge \lceil \frac{n}{2} \rceil$, then $X$ is isomorphic to one of $\mathbb{P}^n$, $Q^n$, and $G(2,\frac{n}{2}+2)$.
\end{enumerate}
\end{conj}

\begin{ntt}\label{Ntt1}
Let $X$ be a Fano manifold and suppose that $X$ admits an $i$-th order minimal family
$$X \vdash H_1 \vdash H_2 \vdash H_3 \vdash \cdots \vdash H_i.$$
We denote by $L_i$ the polarization associated to $H_i$ as in Definition \ref{Def2}.
For $i \ge 2$, we denote by $a_i$ the degree of rational curves parametrized by $H_i$ with respect to $L_{i-1}$.
We denote by $T^i$ the composition of $T$'s associated to $H_1, H_2, H_3, \ldots ,H_i$ in Definition \ref{Def2}:
$$T^i: N^k(X)_{\mathbb{R}} \xrightarrow{T} N^{k-1}(H_1)_{\mathbb{R}} \xrightarrow{T} N^{k-2}(H_2)_{\mathbb{R}} \xrightarrow{T} \ \cdots \ \xrightarrow{T} N^{k-i}(H_i)_{\mathbb{R}}.$$
\end{ntt}

\begin{prop}\label{Prop1}
Let $X$ be a Fano manifold and $m$ a positive integer.
Suppose that $X$ admits an $m$-th order minimal family
$$X \vdash H_1 \vdash H_2 \vdash H_3 \vdash \cdots \vdash H_m.$$
Let $a_i$ be as in Notation \ref{Ntt1}.
\begin{enumerate}
\setlength{\itemsep}{5pt}
\item We assume $a_2 = a_3 = \cdots = a_m =1$.
Then $X$ is covered by rational varieties of dimension $m$.
\item We assume that $H_1$ parametrizes rational curves of degree $1$ with respect to some ample line bundle on $X$ and $a_2 = a_3 = \cdots = a_{m-1} =1$.
Then $X$ is covered by rational varieties of dimension $m$.
\item We assume that $H_1$ parametrizes rational curves of degree $1$ with respect to some ample line bundle on $X$ and $a_2 = a_3 = \cdots = a_m =1$.
Then $X$ is covered by projective spaces of dimension $m$.
\end{enumerate}
\end{prop}

\begin{proof}
We can prove these statements by repeatedly applying Lemma \ref{Lem2} as in the proof of \cite[Theorem 5.1(2)]{Su2} and the proof of \cite[Theorem 1.5]{AC1}.
Here the assumption $a_i=1$ implies ${H_i}^{\circ}=H_i$ and the assumption that $H_1$ parametrizes rational curves of degree $1$ implies ${H_1}^{\circ}=H_1$.
\end{proof}

\section{Computations of Chern characters}

\begin{prop}\label{Prop2}
We define $b_{(i,j,k)} \in \mathbb{Q}\ (1 \le i,\ 1 \le j,\ 1 \le k \le i+j)$ as follows:
\begin{eqnarray*}
b_{(1,j,k)} &:=& \frac{(-1)^{j+1-k} B_{j+1-k}}{(j+1-k)!},\\
b_{(i,j,k)} &:=& \sum_{m=0}^{{\rm min}\{j,\,i+j-k\}} \frac{(-1)^m B_m}{m!} b_{(i-1,j+1-m,k)}\ \ \text{if}\ i\ge2,
\end{eqnarray*}
where $B_m$'s are the Bernoulli numbers (see Definition \ref{Def4} below). 
Let $H_i$, $L_i$, $a_i$, and $T^i$ be as in Notation \ref{Ntt1}.
We assume $a_2 = a_3 = \cdots = a_i =1$.
Then the $j$-th Chern character of $H_i$ is given by the following formula:
\begin{eqnarray*}
{\rm ch}_j(H_i) &=& -\frac{i}{j!}c_1(L_i)^j + \sum_{k=1}^{i} b_{(i,j,k)} T^k\bigl({\rm ch}_k(X)\bigr) c_1(L_i)^j\\
& & + \sum_{k=i+1}^{i+j} b_{(i,j,k)} T^i\bigl({\rm ch}_k(X)\bigr) \cdot c_1(L_i)^{i+j-k}.
\end{eqnarray*}
In particular,
$$c_1(H_1) = \frac{{\rm dim}{H_1}}{2} c_1(L_1)+T\bigl({\rm ch}_2(X)\bigr).$$
\end{prop}

\begin{proof}
If $i=1$, this is \cite[Proposition 1.3]{AC1}.
If $i \ge 2$, we obtain this formula by repeatedly using \cite[Proposition 1.3]{AC1} and Lemma \ref{Lem1}.
See \cite[Proposition 4.4]{Su2}.
\end{proof}

\begin{dfn}\label{Def4}
The Bernoulli numbers $B_m$'s are defined by the following formula:
$$\frac{t}{e^t-1} = \sum_{m=0}^{\infty} \frac{B_m}{m!} t^m.$$
\end{dfn}

\begin{lem}\label{Lem3}
For positive integers $k, n$, we denote by $P_{(k,n)}$ the set of $k$-tuples of positive integers $(l_1, l_2, \ldots , l_k)$ such that $l_1 + l_2 + \cdots + l_k = n$.
(Notice that $P_{(k,n)} = \emptyset$ if $k > n$.)
Let $b_{(i,j,k)}$ be as in Proposition \ref{Prop2}.
Then we have the following formulas:
\begin{enumerate}
\item For any positive integers $i, k$ satisfying $k \le i+1$,
$$b_{(i,1,k)} = \sum_{(l_1, l_2, \ldots , l_k) \in P_{(k,i+1)}} \frac{1}{l_1 l_2 \cdots l_k}.$$
\item For any positive integers $i, k$ satisfying $k \le i+2$,
$$b_{(i,2,k)} = \sum_{(l_1, l_2, \ldots , l_k) \in P_{(k,i+2)}} \frac{1}{l_1 l_2 \cdots l_k} - \sum_{(l_1, l_2, \ldots , l_k) \in P_{(k,i+1)}} \frac{1}{2l_1 l_2 \cdots l_k}.$$
\end{enumerate}
\end{lem}

\begin{proof}
See the proof of \cite[Theorem 5.1]{Na}.
\end{proof}

\begin{lem}\label{Lem4}
For non-negative integers $l \le m$, we denote by $e_l(t_1,t_2,\ldots ,t_m)$ the $l$-th elementary symmetric polynomial in $m$-variables $t_1,t_2, \ldots ,t_m$, namely,
\begin{eqnarray*}
e_0(t_1,t_2,\ldots ,t_m) &:=& 1,\\
e_1(t_1,t_2,\ldots ,t_m) &:=& \sum_{1 \le i \le m} t_i,\\
e_2(t_1,t_2,\ldots ,t_m) &:=& \sum_{1 \le i < j \le m} t_i t_j,\\
&\vdots&\\
e_m(t_1,t_2,\ldots ,t_m) &:=& t_1 t_2 \cdots t_m.
\end{eqnarray*}
Let $k \le n$ be positive integers and $P_{(k,n)}$ as in Lemma \ref{Lem3}.
Then we have the following formula:
$$\sum_{(l_1, l_2, \ldots , l_k) \in P_{(k,n)}} \frac{1}{l_1 l_2 \cdots l_k} = \frac{k!}{n!} e_{n-k}(1,2,\ldots ,n-1)$$
\end{lem}

\begin{proof}
We use by induction on $k$.
If $k=1$, this is clear.
If $k \ge 2$, then we obtain the conclusion from the induction hypothesis as follows:
\begin{eqnarray*}
& & \sum_{(l_1, l_2, \ldots , l_k) \in P_{(k,n)}} \frac{1}{l_1 l_2 \cdots l_k}\\
&=& \frac{1}{n} \sum_{(l_1, l_2, \ldots , l_k) \in P_{(k,n)}} \frac{l_1+ l_2 + \cdots +l_k}{l_1 l_2 \cdots l_k}\\
&=& \frac{1}{n} \sum_{(l_1, l_2, \ldots , l_k) \in P_{(k,n)}} \Bigl( \frac{1}{l_2 l_3 \cdots l_k} + \frac{1}{l_1 l_3 \cdots l_k} + \cdots + \frac{1}{l_1 l_2 \cdots l_{k-1}} \Bigr)\\
&=& \frac{k}{n} \sum_{m=k-1}^{n-1} \sum_{(l_1, l_2, \ldots , l_{k-1}) \in P_{(k-1,m)}} \frac{1}{l_1 l_2 \cdots l_{k-1}}\\
&=& \frac{k}{n} \sum_{m=k-1}^{n-1} \frac{(k-1)!}{m!} e_{m-k+1}(1,2,\ldots ,m-1)\\
&=& \frac{k!}{n!} \sum_{m=0}^{n-k} \frac{(n-1)!}{(m+k-1)!} e_{m}(1,2,\ldots ,m+k-2)\\
&=& \frac{k!}{n!} e_{n-k}(1,2,\ldots ,n-1).
\end{eqnarray*}
\end{proof}

\begin{prop}\label{Prop3}
Let $b_{(i,j,k)}$ be as in Proposition \ref{Prop2}.
Let $i$ be a positive integer.
\begin{enumerate}
\item The following formula holds in $\mathbb{Q}[t]$:
$$\sum_{k=1}^{i+1} \frac{b_{(i,1,k)}}{k!} t^k = \frac{1}{(i+1)!} t(t+1)(t+2) \cdots (t+i).$$
In particular, $b_{(i,1,k)}>0$ for every $1 \le k \le i+1$ and
$$b_{(i,1,i+1)}=1,\ \ \sum_{k=1}^{i+1} \frac{b_{(i,1,k)}}{k!} = 1,\ \ \sum_{k=1}^{i+1} \frac{b_{(i,1,k)}}{k!} 2^k = i+2.$$
\item The following formula holds in $\mathbb{Q}[t]$:
$$\sum_{k=1}^{i+2} \frac{b_{(i,2,k)}}{k!} t^k = \frac{1}{(i+2)!} t(t+1)(t+2) \cdots (t+i)(t+\frac{i}{2}).$$
In particular, $b_{(i,2,k)}>0$ for every $1 \le k \le i+2$ and
$$b_{(i,2,i+2)}=1,\ \ \sum_{k=1}^{i+2} \frac{b_{(i,2,k)}}{k!} = \frac{1}{2},\ \ \sum_{k=1}^{i+2} \frac{b_{(i,2,k)}}{k!} 2^k = \frac{i+4}{2}.$$
\end{enumerate}
\end{prop}

\begin{proof}
These formulas follow immediately from Lemmas \ref{Lem3} and \ref{Lem4}:
\begin{eqnarray*}
\sum_{k=1}^{i+1} \frac{b_{(i,1,k)}}{k!} t^k &=& \frac{1}{(i+1)!} \sum_{k=1}^{i+1} e_{i-k+1}(1,2,\ldots ,i) t^k\\
&=& \frac{1}{(i+1)!} t(t+1)(t+2) \cdots (t+i),
\end{eqnarray*}
\begin{eqnarray*}
\sum_{k=1}^{i+2} \frac{b_{(i,2,k)}}{k!} t^k &=& \frac{1}{(i+2)!} \sum_{k=1}^{i+2} e_{i-k+2}(1,2,\ldots ,i+1) t^k\\
& & -  \frac{1}{2(i+1)!} \sum_{k=1}^{i+1} e_{i-k+1}(1,2,\ldots ,i) t^k\\
&=& \frac{1}{(i+2)!} t(t+1)(t+2) \cdots (t+i)(t+i+1)\\
& & - \frac{1}{2(i+1)!} t(t+1)(t+2) \cdots (t+i)\\
&=& \frac{1}{(i+2)!} t(t+1)(t+2) \cdots (t+i)(t+\frac{i}{2}).
\end{eqnarray*}
We remark that $b_{(i,1,k)}>0$ and $b_{(i,2,k)}>0$ have been proved in \cite[Theorem 5.1]{Na}.
\end{proof}

\section{Proof of theorems}

\begin{thm}\label{Thm4}
Let $X$ be a Fano manifold and $m$ a positive integer.
Assume that
$${\rm ch}_{k}(X) \ge \frac{m+1}{k!}$$
for every $1 \le k \le m$.
Then the following statements hold:
\begin{enumerate}
\setlength{\itemsep}{3pt}
\item $\underline{N}_X \ge m$.
\item $X$ is covered by rational varieties of dimension $m$.
Moreover, if $X$ is covered by rational curves of degree $1$ with respect to some ample line bundle on $X$, then $X$ is covered by projective spaces of dimension $m$.
\end{enumerate}
\end{thm}

\begin{proof}
We use Notation \ref{Ntt1}.
The result is clear if $m=1$, so we assume $m \ge 2$.

\begin{claim}\label{Claim1}
For any minimal family $X \vdash H_1$, the following statements hold:
\begin{enumerate}
\setlength{\itemsep}{3pt}
\item[{\rm (i)}] $H_1$ is also a Fano manifold.
\item[{\rm (ii)}] $a_2=1$ for any second order minimal family $H_1 \vdash H_2$.
\end{enumerate}
\end{claim}

\begin{proof}
(i) follows immediately from Lemma \ref{Lem1} and Proposition \ref{Prop2}:
$${\rm dim}{H_1} = T\bigl(c_1(X)\bigr)-2 \ge m-1> 0$$
and
$$c_1(H_1) = \frac{{\rm dim}{H_1}}{2} c_1(L_1)+T\bigl({\rm ch}_2(X)\bigr) >0.$$
Moreover, Lemma \ref{Lem1} yields
$${\rm dim}H_2 = T\bigl(c_1(H_1)\bigr)-2 = \frac{{\rm dim}H_1}{2}a_2 + T^2\bigl({\rm ch}_2(X)\bigr)-2$$
and
$$T^2\bigl({\rm ch}_2(X)\bigr) \ge \frac{m+1}{2} \ge \frac{3}{2}.$$
This implies (ii) because ${\rm dim}H_2 < {\rm dim}H_1$.
\end{proof}

\begin{claim}\label{Claim2}
Let $i$ be a positive integer satisfying $2 \le i < m$.
We assume that $X$ admits an $i$-th order minimal family
$$X \vdash H_1 \vdash H_2 \vdash H_3 \vdash \cdots \vdash H_i.$$ 
We also assume $a_2 = a_3 = \cdots = a_i = 1$.
Then the following statements hold:
\begin{enumerate}
\setlength{\itemsep}{3pt}
\item[{\rm (i)}] $H_i$ is also a Fano manifold.
\item[{\rm (ii)}] $a_{i+1}=1$ for any $(i+1)$-th order minimal family $H_i \vdash H_{i+1}$.
\end{enumerate}
\end{claim}

\begin{proof}
We obtain (i) by using Lemma \ref{Lem1}, Propositions \ref{Prop2}, and \ref{Prop3}(1) as follows:
\begin{eqnarray*}
{\rm dim}{H_i} &=& T\bigl(c_1(H_{i-1})\bigr)-2\\
&=& T \Bigl( -(i-1)c_1(L_{i-1}) + \sum_{k=1}^{i-1} b_{(i-1,1,k)} T^k\bigl({\rm ch}_k(X)\bigr)c_1(L_{i-1})\\
& & + b_{(i-1,1,i)} T^{i-1}\bigl({\rm ch}_i(X)\bigr) \Bigr) -2\\
&=& -(i-1) + \sum_{k=1}^{i} b_{(i-1,1,k)} T^k\bigl({\rm ch}_k(X)\bigr) -2\\
&\ge& -(i-1) + \sum_{k=1}^{i} b_{(i-1,1,k)} \frac{m+1}{k!} -2\\
&=& -(i-1) + (m+1) -2 \ > \ 0
\end{eqnarray*}
and
\begin{eqnarray*}
c_1(H_i) &=& -ic_1(L_i) + \sum_{k=1}^i b_{(i,1,k)} T^k\bigl({\rm ch}_k(X)\bigr)c_1(L_i) + b_{(i,1,i+1)} T^i\bigl({\rm ch}_{i+1}(X)\bigr)\\
&\ge& \Bigl( -i + \sum_{k=1}^i b_{(i,1,k)} \frac{m+1}{k!} \Bigr)c_1(L_i) + b_{(i,1,i+1)} T^i\bigl({\rm ch}_{i+1}(X)\bigr)\\
&=&\Biggl( -i + \Bigl( 1-\frac{1}{(i+1)!} \Bigr) (m+1) \Biggr)c_1(L_i) + T^i\big({\rm ch}_{i+1}(X)\bigr) \ > \ 0,
\end{eqnarray*}
where
$$-i + \Bigl( 1-\frac{1}{(i+1)!} \Bigr) (m+1) > -i + \Bigl( 1-\frac{1}{(i+1)!} \Bigr) (i+1) = 1-\frac{1}{i!} > 0.$$
Moreover, Lemma \ref{Lem1}, Propositions \ref{Prop2}, and \ref{Prop3}(2) yield
$${\rm dim}{H_{i+1}} = T\bigl(c_1(H_i)\bigr)-2 = \frac{{\rm dim}H_i}{2}a_{i+1} + T^2\bigl({\rm ch}_2(H_{i-1})\bigr)-2$$
and
\begin{eqnarray*}
T^2\bigl({\rm ch}_2(H_{i-1})\bigr) &=& T^2 \Bigl( -\frac{i-1}{2}c_1(L_{i-1})^2 + \sum_{k=1}^{i-1} b_{(i-1,2,k)} T^k\bigl({\rm ch}_k(X)\bigr)c_1(L_{i-1})^2\\
& & + b_{(i-1,2,i)} T^{i-1}\bigl({\rm ch}_i(X)\bigr) \cdot c_1(L_{i-1}) + b_{(i-1,2,i+1)} T^{i-1}\bigl({\rm ch}_{i+1}(X)\bigr) \Bigr)\\
&=& -\frac{i-1}{2} a_{i+1} +  \sum_{k=1}^i b_{(i-1,2,k)} T^k\bigl({\rm ch}_k(X)\bigr) a_{i+1} + T^{i+1}\bigl({\rm ch}_{i+1}(X)\bigr)\\
&\ge& \Bigl( -\frac{i-1}{2} + \sum_{k=1}^i b_{(i-1,2,k)} \frac{m+1}{k!} \Bigr) a_{i+1} + \frac{m+1}{(i+1)!}\\
&=& \Bigl( -\frac{i-1}{2} + \bigl(\frac{1}{2}-\frac{1}{(i+1)!}\bigr) (m+1) \Bigr) a_{i+1} + \frac{m+1}{(i+1)!}\\
&\ge& \Bigl( -\frac{i-1}{2} + \bigl(\frac{1}{2}-\frac{1}{(i+1)!}\bigr) (m+1) \Bigr) + \frac{m+1}{(i+1)!}\\
&=& \frac{m-i+2}{2} \ \ge \ \frac{3}{2},
\end{eqnarray*}
where
$$-\frac{i-1}{2} + \bigl(\frac{1}{2}-\frac{1}{(i+1)!}\bigr) (m+1) > -\frac{i-1}{2} + \bigl(\frac{1}{2}-\frac{1}{(i+1)!}\bigr) (i+1) = 1-\frac{1}{i!} > 0.$$
This implies (ii) because ${\rm dim}H_{i+1} < {\rm dim}H_i$.
\end{proof}

Now, (1) follows from Claims \ref{Claim1} and \ref{Claim2}.
Moreover, any chain of length $m$
$$X \vdash H_1 \vdash H_2 \vdash H_3 \vdash \cdots \vdash H_m$$
satisfies $a_2 = a_3 = \cdots =a_m=1$.
This implies (2) by applying Proposition \ref{Prop1}.
\end{proof}

\begin{thm}\label{Thm5}
Let $X$ be a Fano manifold and $m$ a positive integer.
Assume that $X$ is covered by rational curves of degree $1$ with respect to some ample line bundle $L$ on $X$ and
$${\rm ch}_{k}(X) \ge \frac{2m+1-2^k}{k!} c_1(L)^k$$
for every $1 \le k \le m$.
Then the following statements hold:
\begin{enumerate}
\setlength{\itemsep}{3pt}
\item $\overline{N}_X \ge m$.
Moreover, if every minimal family parametrizes rational curves of degree $1$ with respect to $L$, then $\underline{N}_X \ge m$.
\item $X$ is covered by rational varieties of dimension $m$ and by projective spaces of dimension $m-1$.
Moreover, under the stronger assumption that
$${\rm ch}_{k}(X) \ge \frac{2m+2-2^k}{k!} c_1(L)^k$$
for every $1 \le k \le m$, $X$ is covered by projective spaces of dimension $m$.
\end{enumerate}
\end{thm}

\begin{proof}
We use Notation \ref{Ntt1}.
The result is clear if $m=1$, so we assume $m \ge 2$.

\begin{claim}\label{Claim3}
For any minimal family $X \vdash H_1$, the following statements hold:
\begin{enumerate}
\setlength{\itemsep}{3pt}
\item[{\rm (i)}] $H_1$ is also a Fano manifold.
\item[{\rm (ii)}] If $2<m$, then $a_2=1$ for any second order minimal family $H_1 \vdash H_2$.
\end{enumerate}
\end{claim}

\begin{proof}
(i) follows immediately from Lemma \ref{Lem1} and Proposition \ref{Prop2}:
$${\rm dim}{H_1} = T\bigl(c_1(X)\bigr)-2 \ge 2m-3 > 0$$
and
$$c_1(H_1) = \frac{{\rm dim}{H_1}}{2} c_1(L_1)+T\bigl({\rm ch}_2(X)\bigr) >0.$$
Moreover, Lemma \ref{Lem1} yields
$${\rm dim}H_2 = T\bigl(c_1(H_1)\bigr)-2 = \frac{{\rm dim}H_1}{2} a_2 + T^2\bigl({\rm ch}_2(X)\bigr)-2$$
and
$$T^2\bigl({\rm ch}_2(X)\bigr) \ge \frac{2m-3}{2} a_2 > a_2$$
if $2<m$.
This implies (ii) because ${\rm dim}H_2 < {\rm dim}H_1$.
\end{proof}

\begin{claim}\label{Claim4}
Let $i$ be a positive integer satisfying $2 \le i < m$.
We assume that $X$ admits an $i$-th order minimal family
$$X \vdash H_1 \vdash H_2 \vdash H_3 \vdash \cdots \vdash H_i.$$ 
We also assume that $H_1$ parametrizes rational curves of degree $1$ with respect to $L$ and $a_2 = a_3 = \cdots = a_i = 1$.
Then the following statements hold:
\begin{enumerate}
\setlength{\itemsep}{3pt}
\item[{\rm (i)}] $H_i$ is also a Fano manifold.
\item[{\rm (ii)}] If $i+1<m$, then $a_{i+1}=1$ for any $(i+1)$-th order minimal family $H_i \vdash H_{i+1}$.
\end{enumerate}
\end{claim}

\begin{proof}
We obtain (i) by using Lemma \ref{Lem1}, Propositions \ref{Prop2}, and \ref{Prop3}(1) as follows:
\begin{eqnarray*}
{\rm dim}{H_i} &=& T\bigl(c_1(H_{i-1})\bigr)-2\\
&=& T \Bigl( -(i-1)c_1(L_{i-1}) + \sum_{k=1}^{i-1} b_{(i-1,1,k)} T^k\bigl({\rm ch}_k(X)\bigr)c_1(L_{i-1})\\
& & + b_{(i-1,1,i)} T^{i-1}\bigl({\rm ch}_i(X)\bigr) \Bigr) -2\\
&=& -(i-1) + \sum_{k=1}^{i} b_{(i-1,1,k)} T^k\bigl({\rm ch}_k(X)\bigr) -2\\
&\ge& -(i-1) + \sum_{k=1}^{i} b_{(i-1,1,k)} \frac{2m+1-2^k}{k!} -2\\
&=& -(i-1) + (2m+1) - (i+1) -2 >0
\end{eqnarray*}
and
\begin{eqnarray*}
c_1(H_i) &=& -ic_1(L_i) + \sum_{k=1}^i b_{(i,1,k)} T^k\bigl({\rm ch}_k(X)\bigr)c_1(L_i) + b_{(i,1,i+1)} T^i\bigl({\rm ch}_{i+1}(X)\bigr)\\
&\ge& \Bigl( -i + \sum_{k=1}^{i+1} b_{(i,1,k)} \frac{2m+1-2^k}{k!} \Bigr)c_1(L_i)\\
&=& \Bigl( -i + (2m+1) - (i+2) \Bigr)c_1(L_i) > 0.
\end{eqnarray*}
Moreover, Lemma \ref{Lem1}, Propositions \ref{Prop2}, and \ref{Prop3}(2) yield
$${\rm dim}{H_{i+1}} = T\bigl(c_1(H_i)\bigr)-2 = \frac{{\rm dim}H_i}{2}a_{i+1} + T^2\bigl({\rm ch}_2(H_{i-1})\bigr)-2$$
and
\begin{eqnarray*}
T^2\bigl({\rm ch}_2(H_{i-1})\bigr) &=& T^2 \Bigl( -\frac{i-1}{2}c_1(L_{i-1})^2 + \sum_{k=1}^{i-1} b_{(i-1,2,k)} T^k\bigl({\rm ch}_k(X)\bigr)c_1(L_{i-1})^2\\
& & + b_{(i-1,2,i)} T^{i-1}\bigl({\rm ch}_i(X)\bigr) \cdot c_1(L_{i-1}) + b_{(i-1,2,i+1)} T^{i-1}\bigl({\rm ch}_{i+1}(X)\bigr) \Bigr)\\
&=& -\frac{i-1}{2} a_{i+1} +  \sum_{k=1}^i b_{(i-1,2,k)} T^k\bigl({\rm ch}_k(X)\bigr) a_{i+1}\\
& & + b_{(i-1,2,i+1)} T^{i+1}\bigl({\rm ch}_{i+1}(X)\bigr)\\
&\ge& \Bigl( -\frac{i-1}{2} +  \sum_{k=1}^{i+1} b_{(i-1,2,k)} \frac{2m+1-2^k}{k!} \Bigr) a_{i+1}\\
&=& \Bigl( -\frac{i-1}{2} + \frac{1}{2} (2m+1) - \frac{i+3}{2} \Bigr) a_{i+1} > a_{i+1}
\end{eqnarray*}
if $i+1<m$.
This implies (ii) because ${\rm dim}H_{i+1} < {\rm dim}H_i$.
\end{proof}

Now, (1) follows from Claims \ref{Claim3} and \ref{Claim4}.
Moreover, we can take a chain of length $m$
$$X \vdash H_1 \vdash H_2 \vdash H_3 \vdash \cdots \vdash H_m$$
such that $H_1$ parametrizes rational curves of degree $1$ with respect to $L$ and $a_2 = a_3 = \cdots =a_{m-1}=1$.
This implies the former statement of (2) by applying Proposition \ref{Prop1}.

Finally, we consider the case that the Chern characters of $X$ satisfy the assumption of the latter statement of (2).
Then we can remove the assumption $2<m$ from Claim \ref{Claim3}(ii) because
$$T^2\bigl({\rm ch}_2(X)\bigr) \ge (m-1) a_2 \ge a_2,$$
and we can also remove the assumption $i+1<m$ from Claim \ref{Claim4}(ii) because
$$T^2\bigl({\rm ch}_2(H_{i-1})\bigr) \ge \Bigl( -\frac{i-1}{2} + \frac{1}{2} (2m+2) - \frac{i+3}{2} \Bigr) a_{i+1} \ge a_{i+1}.$$
Thus we obtain $a_m=1$.
This implies that $X$ is covered by projective spaces of dimension $m$ by applying Proposition \ref{Prop1}.
\end{proof}

\end{document}